\documentclass[12pt]{elsarticle}
\usepackage[a4paper, top=2.1cm, bottom=2.1cm, left=2.3cm, right=2.3cm]{geometry}
\usepackage{amsmath}
 \usepackage{graphicx}
\usepackage{amsthm}
\usepackage{color}
\usepackage{fancyhdr}
\usepackage{pdfpages}
\usepackage{txfonts}
\usepackage{eucal}
\usepackage{paralist}
\usepackage{extarrows}
\usepackage{setspace}
\usepackage{braket} 
\allowdisplaybreaks
\usepackage[colorlinks=true,citecolor=red,linkcolor=red,anchorcolor=red,urlcolor=red]{hyperref}
\allowdisplaybreaks
\newtheorem{theorem}{Theorem}[section]

\newtheorem{lemma}{Lemma}[section]

\numberwithin{equation}{section}
\makeatletter
\def\ps@pprintTitle{%
  \let\@oddhead\@empty
 \let\@evenhead\@empty
  \def\@oddfoot{\reset@font\hfil\thepage\hfil}
 \let\@evenfoot\@oddfoot
}
\makeatother

\begin{document}
\begin{frontmatter}
\title{\textbf{Standing waves for the Chern-Simons-Schr\"{o}dinger equation with critical exponential growth}}

\author{Chao Ji $^{\mbox{a},}$\footnote[2]{ E-mail: jichao@ecust.edu.cn}}
\author{Fei Fang $^{\mbox{b},}$\footnote[1]{Corresponding author. E-mail: fangfei68@163.com}{} }

\address{ $^{a}$ Department of Mathematics, East China University of Science and Technology\\ Shanghai, 200237,  China}

\address{ $^{\mbox{b}}$ Department of Mathematics, Beijing Technology and Business University\\ Beijing, 100048, China}

\begin{abstract}
In this paper, by combing the variational methods and Trudinger-Moser inequality, we study the existence and multiplicity of the positive standing wave for the following Chern-Simons-Schr\"{o}dinger equation
\begin{equation}
-\Delta u+u +\lambda\left(\int_{0}^{\infty}\frac{h(s)}{s}u^{2}(s)ds+\frac{h^{2}(\vert x\vert)}{\vert x\vert^{2}}\right)u=f(x,u)+\epsilon k(x)\quad\quad \text{ in}\,\,\mathbb{R}^2, \\
\end{equation}
where $h(s)=\int_{0}^{s}\frac{l}{2}u^{2}(l)dl$, $\lambda>0$ and the nonlinearity  $f:\mathbb{R}^2\times \mathbb{R}\rightarrow \mathbb{R}$ behaves like $\text{exp}(\alpha\vert u\vert^{2})$ as
$\vert u\vert\rightarrow \infty$. For the case $\epsilon=0$, we can get a mountain-pass type solution.

\end{abstract}

\end{frontmatter}
\section{Introduction}
In this paper, we are concerned with the following nonlinear Chern-Simons-Schr\"{o}dinger system
\begin{align}
 \left\{
\begin{aligned}
&iD_{0}\phi+(D_{1}D_{1}+D_{2}D_{2})\phi=f(x, \phi),\\
&\partial_{0}A_{1}-\partial_{1}A_{0}=-\text{Im}(\bar{\phi}D_{2}\phi),\\
&\partial_{0}A_{2}-\partial_{2}A_{0}=\text{Im}(\bar{\phi}D_{1}\phi),\\
&\partial_{1}A_{2}-\partial_{2}A_{1}=-\frac{1}{2}\vert \phi\vert^{2},
\end{aligned}
\right.
\end{align}
where $i$ denotes the imaginary unit, $\partial_{0}=\frac{\partial}{\partial t}$, $\partial_{1}=\frac{\partial}{\partial x_{1}}$, $\partial_{1}=\frac{\partial}{\partial x_{2}}$ for $(t, x_{1}, x_{2})\in \mathbb{R}^{1+2}$, $\phi:\mathbb{R}^{1+2}\rightarrow \mathbb{C}$ is the complex
scalar field, $A_{\mu}:\mathbb{R}^{1+2}\rightarrow \mathbb{R}$ is the gauge field, $D_{\mu}=\partial_{\mu}+iA_{\mu}$ is the covariant derivative for $\mu=0, 1, 2$. This system was proposed in \cite{r11, r12} and consists of the Schrodinger equation augmented by the gauge field $A_{\mu}$. As usual in Chern-Simons theory,
this system is invariant under the following gauge transformation
\begin{align*}
 \phi\rightarrow\phi e^{i\chi}, \quad\quad\quad\quad A_{\mu}\rightarrow A_{\mu}-\partial_{\mu}\chi,
\end{align*}
where $\chi:\mathbb{R}^{1+2}\rightarrow \mathbb{R}$ is an arbitrary $C^{\infty}$ function.\\

In recent years, the Chern-Simons-Schr\"{o}dinger systems have received considerable
attention, these models are very important for the study of
the high-temperature superconductor, Aharovnov-Bohm scattering, and quantum Hall effect. In \cite{r4}, the authors investigated the system
(1.1) with power type nonlinearity, that is $f(x, u)=\lambda\vert u\vert^{p-2}u \,(\text{here}\,\, p>2, \lambda>0)$ and sought
the standing waves solutions to the system (1.1) of the form
\begin{align}
&\phi(t, x)=u(\vert x\vert)e^{i\omega t}, \quad\quad\quad A_{0}(x, t)=k(\vert x\vert),\nonumber\\
&A_{1}(x, t)=\frac{x_{2}}{\vert x\vert}h(\vert x\vert),\quad\quad\quad A_{2}(x, t)=-\frac{x_{1}}{\vert x\vert}h(\vert x\vert),
\end{align}
where $\omega>0$ is a given frequency and $u$, $k$ and $h$ are real value functions on $[0, +\infty)$ such that $h(0)=0$. Note that the ansatz
(1.2) satisfies the Coulomb gauge condition $\partial_{1}A_{1}+\partial_{2}A_{2}=0$. Inserting the
ansatz (1.2) into the system (1.1), the authors in \cite{r4} got the following nonlocal
semilinear elliptic equation for $u$
\begin{equation}
-\Delta u+\omega u +\left(\xi+\int_{\vert x\vert}^{\infty}\frac{h(s)}{s}u^{2}(s)ds\right)u+\frac{h^{2}(\vert x\vert)}{\vert x\vert^{2}}u=\lambda\vert u\vert^{p-2}u\quad\quad \text{ in}\,\,\mathbb{R}^{2},
\end{equation}
where $h(s)=\int_{0}^{s}\frac{l}{2}u^{2}(l)dl$ and $\xi$ is a constant. Moreover, they showed (1.3) is actually the Euler-Lagrange equation of the following functional
\begin{equation*}
I(u)=\frac{1}{2}\int_{\mathbb{R}^{2}}\vert \nabla u\vert^{2}+(\omega+\xi)u^{2}+\frac{u^{2}}{\vert x\vert^{2}}\left(\int_{0}^{\vert x\vert}\frac{s}{2}u^{2}(s)\right)^{2}dx-\frac{\lambda}{p}\int_{\mathbb{R}^{2}}\vert u\vert^{p}dx, \quad\quad u\in H^{1}_{r}(\mathbb{R}^{2})
\end{equation*}
here $H^{1}_{r}(\mathbb{R}^{2})$ denotes the set of radially symmetric functions in $H^{1}(\mathbb{R}^{2})$. They also showed that $I\in C^{1}(H^{1}_{r}(\mathbb{R}^{2}), \mathbb{R})$ and established some
existence results of standing waves by applying variational methods.\\

Another interesting result is in \cite{r15}, the authors studied whether $I$ is bounded from below or not for $p\in (1, 3)$. They proved the existence of a threshold value $\omega_{0}$ such that $I$ is bounded from below if $\omega\geq \omega_{0}$, and it is not for $\omega\in (0, \omega_{0})$. In fact, they given an explicit expression of $\omega_{0}$, namely:
\begin{equation*}
\omega_{0}=\frac{3-p}{3+p}3^{\frac{p-1}{2(3-p)}}2^{\frac{2}{3-p}}\Bigg(\frac{m^{2}(3+p)}{p-1}\Bigg)^{-\frac{p-1}{2(3-p)}},
\end{equation*}
with
\begin{equation*}
m=\int_{-\infty}^{+\infty}\left(\frac{2}{p+1}\coth^{2}\left(\frac{p-1}{2}r\right)\right)^{\frac{2}{1-p}}dr.
\end{equation*}
Moreover, in \cite{r6}, the authors studied the Chern-Simons-Schr\"{o}dinger system with the general nonlinearity which is a Berestycki, Gallou\"{e}t and Kavian type nonlinearity \cite{r1} and it is the planar version of the Berestycki-Lions type nonlinearity \cite{r2, r3}. In \cite{r18}, the authors researched the Chern-Simons-Schr\"{o}dinger system
without Ambrosetti-Rabinowitz condition. The other related research for system (1.1), we may refer to \cite{r13, r16, r20}. However, to our knowledge, the Chern-Simons-Schr\"{o}dinger system with critical exponential growth was not considered until now, that is, $f$ behaves like $\text{exp}(\alpha\vert u\vert^{2})$ as $\vert u\vert\rightarrow \infty$. More precisely, there exists $\alpha_{0}>0$ such that
\begin{equation*}
\underset{\vert s\vert\rightarrow \infty}{\lim}\frac{\vert f(x,s)\vert}{e^{\alpha s^{2}}}=0, \quad \forall \alpha>\alpha_{0},\quad \text{and}\quad \underset{\vert s\vert\rightarrow \infty}{\lim}\frac{\vert f(x,s)\vert}{e^{\alpha s^{2}}}=+\infty, \quad \forall \alpha<\alpha_{0}.
\end{equation*}
In order to study this class of problems,  the Trudinger-Moser inequalities are very important. If $\Omega$ be a bounded domain in $\mathbb{R}^{2}$, the authors in \cite{r14, r17} asserts that
\begin{equation*}
\exp(\alpha\vert u\vert^{2})\in L^{1}(\Omega), \quad\quad \forall u\in H_{0}^{1}(\Omega), \quad \alpha>0,
\end{equation*}
and there exists a constant $C>0$ such that
\begin{equation*}
\underset{\vert \nabla u\vert_{L^{2}(\Omega)}\leq 1}{\sup}\int_{\Omega}\exp(\alpha\vert u\vert^{2})\leq C<\infty , \quad\quad \text{if}\,\, \alpha\leq 4\pi.
\end{equation*}
Afterwards, Cao in \cite{r5} proved a version of Trudinger-Moser inequality in whole space in $\mathbb{R}^{2}$, which was improved by do \'{O} in \cite{r8}
\begin{equation}
\int_{\mathbb{R}^{2}}\Big(\exp(\alpha\vert u\vert^{2})-1\Big)dx<+\infty, \quad \forall u\in H^{1}(\mathbb{R}^{2}), \, \alpha>0.
\end{equation}
Moreover, if $\alpha< 4\pi$ and $\vert u\vert_{L^{2}(\mathbb{R}^{2})}\leq M$, then there exists a constant $C=C(M, \alpha)>0$ which depends only on
$M$ and $\alpha$, such that
\begin{equation}
\underset{\vert \nabla u\vert_{L^{2}(\mathbb{R}^{2})}\leq 1}{\sup}\int_{\mathbb{R}^{2}}\Big(\exp(\alpha\vert u\vert^{2})-1\Big)dx\leq C.
\end{equation}

In this paper, we will firstly study the existence of positive solution of the equation without the perturbation
\begin{equation}
-\Delta u+u +\lambda\left(\int_{0}^{\infty}\frac{h(s)}{s}u^{2}(s)ds+\frac{h^{2}(\vert x\vert)}{\vert x\vert^{2}}\right)u=f(x,u)\quad\quad \text{ in}\,\,\mathbb{R}^{2}. \\
\end{equation}
Combing Trudinger-Moser inequalities (1.4), (1.5) and mountain pass theorem, there exists $\lambda_{1}>0$, such that for any $0<\lambda<\lambda_{1}$,
we can get a positive and classical mountain-pass type solution.\\

Our next concern is problem (0.1). When the positive parameter $\lambda$ and $\epsilon$ are small enough, we can find a mountain-pass solution. Moreover, by combing Trudinger-Moser inequality and Ekeland's variational principle \cite{r10}, we can find a local minimal solution with negative energy.\\

Since we are interested in the positive solutions, we may assume $f:\mathbb{R}^{2}\times \mathbb{R}\rightarrow \mathbb{R}$ is continuous and
 $f(x, s)=0$ for $\mathbb{R}^{2}\times (-\infty, 0)$. Moreover, we assume the following growth conditions on the nonlinearity $f(x, s)$:\\
$(f_{1})$ $f(x, s)\leq Ce^{4\pi s^{2}}$, for all $(x, s)\in \mathbb{R}^{2}\times [0, \infty)$;\\
$(f_{2})$ $\underset{s\rightarrow 0}{\lim}\frac{f(x, s)}{s}=0$ uniformly with respect to  $x\in \mathbb{R}^{2}$;\\
$(f_{3})$ There is $0\leq\sigma<2$ such that
\begin{equation*}
 sf(x, s)-6F(x, s)\geq -\sigma s^{2}, \quad \text{for all}\, (x, s)\in \mathbb{R}^{2}\times [0, \infty),
\end{equation*}
where $F$ is the primitive of $f$.\\
$(f_{4})$ There exist constants $p>6$ and $C_{p}>0$ such that \\
\begin{equation*}
f(x, s)\geq C_{p}s^{p-1}, \quad\quad \text{for all}\,(x, t)\in \mathbb{R}^{2}\times[0, +\infty),
\end{equation*}
where
\begin{equation*}
C_{p}>\left[\frac{6(p-2)}{p(2-\sigma)}\right]^{\frac{p-2}{2}}S_{p}^{p}
\end{equation*}
and
\begin{equation*}
S_{p}=\underset{u\in H^{1}_{r}(\mathbb{R}^{2})\setminus {0}}{\inf}\frac{\Bigg(\int_{\mathbb{R}^{2}}(\vert \nabla u\vert^{2}+\vert u\vert^{2})dx\Bigg)^{\frac{1}{2}}}{\Bigg(\int_{\mathbb{R}^{2}}\vert u\vert^{p}dx\Bigg)^{\frac{1}{p}}}.
\end{equation*}

The following are main results of this paper.
\begin{theorem}\label{gt621}
Suppose hypotheses  $(f_{1})-(f_{4})$ hold, then there exists $\lambda_{1}>0$, such that for any $0<\lambda<\lambda_{1}$, problem
$(1.6)$ has a positive and classical solution of mountain-pass type.
\end{theorem}

\begin{theorem}\label{gt621}
Suppose hypotheses  $(f_{1})-(f_{4})$ hold, for any $0\leq h(x)\in H^{-1}$, then there exist $\lambda_{2}>0$ and $\epsilon_{1}>0$, such that for any $0<\lambda<\lambda_{2}$ and $0<\epsilon<\epsilon_{1}$, problem $(0.1)$ has at least two nonnegative solutions and one of them has a negative energy.
\end{theorem}

The paper is organized as follows. In Section 2 we are
concerned with the nonperturbation problem (1.4) and prove Theorem 1.1.
In Section 3,  the proof of Theorem 1.2 is given.\\

\textbf{Notations}. $C$, $C_{1}$, $C_{2}$ etc. will denote positive constants whose essential values are inessential. $H^{-1}$ is dual space of $H^{1}(\mathbb{R}^{2})$,  $C_{0, r}^{\infty}(\mathbb{R}^{2})$ denotes the space of infinitely differential radial functions with compact support in $\mathbb{R}^{2}$. $o_{n}(1)$ denotes a quantity which goes to zero. $B_{R}$ denotes the open ball centered at the origin and radius $R>0$ and
$\bar{B}_{R}$ is its closure. $u_{n}\rightarrow u$ and $u_{n}\rightharpoonup u$ denote the strong convergence and weak convergence of a sequence $\{u_{n}\}$ in a Banach space, respectively.\\

\section{Proof of Theorem 1.1}

From assumptions $(f_{1})$ and $(f_{2})$, for given $\eta>0$ small there exist positive constants $C_{\eta}$ and $\gamma>1$ such that
\begin{equation*}
 F(x, s)\leq\eta\frac{s^{2}}{2}+C_{\eta}\Big(e^{\gamma\pi s^{2}}-1\Big)\quad\quad\quad \text{for all}\,\, (x, s)\in \mathbb{R}^{2}\times \mathbb{R}.
\end{equation*}
Thus, by the Trudinger-Moser inequalities (1.4), we have $F(x, u)\in L^{1}(\mathbb{R}^{2})$ for all $u\in H^{1}_{r}(\mathbb{R}^{2})$. Therefore,
the functional
\begin{align*}
J(u)&=\frac{1}{2}\int_{\mathbb{R}^{2}}(\vert \nabla u\vert^{2}+u^{2})dx+\frac{\lambda}{2}\int_{\mathbb{R}^{2}}\frac{u^{2}}{\vert x\vert^{2}}\left(\int_{0}^{\vert x\vert}\frac{s}{2}u^{2}(s)\right)^{2}dx-\int_{\mathbb{R}^{2}}F(x, u)dx\\
&=\frac{1}{2}\int_{\mathbb{R}^{2}}(\vert \nabla u\vert^{2}+u^{2})dx+\frac{\lambda}{2}\int_{\mathbb{R}^{2}}\frac{u^{2}}{4\vert x\vert^{2}}\left(\frac{1}{2\pi}\int_{B_{\vert x\vert}}u^{2}\right)^{2}dx-\int_{\mathbb{R}^{2}}F(x, u)dx\\
&=\frac{1}{2}\int_{\mathbb{R}^{2}}(\vert \nabla u\vert^{2}+u^{2})dx+\frac{\lambda}{2}c(u)-\int_{\mathbb{R}^{2}}F(x, u)dx\\
&=J_{0}(u)+\frac{\lambda}{2}c(u)\quad\quad u\in H^{1}_{r}(\mathbb{R}^{2})
\end{align*}
is well defined. Furthermore, using standard arguments (see \cite{r19}) we can show that $J\in C^{1}(H^{1}_{r}(\mathbb{R}^{2}), \mathbb{R})$ with
\begin{equation*}
J'(u)\phi=\int_{\mathbb{R}^{2}} \nabla u\nabla \varphi+u\varphi dx+\lambda\int_{\mathbb{R}^{2}}\left(\int_{\vert x\vert}^{\infty}\frac{h(s)}{s}u^{2}(s)ds\right)u\varphi+\frac{h^{2}(\vert x\vert)}{\vert x\vert^{2}}u\varphi dx-\int_{\mathbb{R}^{2}}f(x, u)\varphi dx \quad\quad \varphi\in H^{1}_{r}(\mathbb{R}^{2}).
\end{equation*}
Consequently, each critical point of the functional $J$ is a solution of problem (1.6).\\

For functional $c(u)$, there is the following compactness lemma we use later.
\begin{lemma}(\emph{see \cite{r4}})
Suppose that a sequence  $\{u_{n}\}$ converges weakly to a function $u$ in $H^{1}_{r}(\mathbb{R}^{2})$ as $n\rightarrow \infty$. Then
for each $\varphi\in H^{1}_{r}(\mathbb{R}^{2})$, $c(u_{n})$, $c'(u_{n})\varphi$ and $c'(u_{n})u_{n}$ converges up to a subsequence to $c(u)$,
$c'(u)\varphi$ and $c'(u)u$, respectively, as $n\rightarrow \infty$.
\end{lemma}

In order to show that the weak limit of a sequence in $H^{1}_{r}(\mathbb{R}^{2})$  is a weak solution of problem (1.6), we need the following convergence
result.
\begin{lemma}(\emph{see \cite{r7}})
Assume that $\Omega \subset \mathbb{R}^{2}$ be a bounded domain and $f:\Omega \times \mathbb{R}\rightarrow \mathbb{R}$ a continuous function. Let $(u_{n})$ be a sequence of functions in $L^{1}(\Omega)$
converging to $u$ in $L^{1}(\Omega)$. Assume that $f(x, u_{n})$ and $f(x, u)$ are also $L^{1}$ functions. If
\begin{equation*}
\int_{\Omega}\vert f(x, u_{n})u_{n}\vert dx\leq C_{1},
\end{equation*}
then $f(x, u_{n})$ converges in $L^{1}$ to $f(x, u)$.
\end{lemma}

In order to construct the mountain-pass geometry of the functional $J$, we need next two lemmas.
\begin{lemma}(\emph{see \cite{r9}})
Let  $\beta>0$ and $r>1$. Then for each $\alpha>r$ there exists a positive constant $C=C(\alpha)$ such that for all $s\in \mathbb{R}$,
\begin{equation*}
\Big(e^{\beta s^{2}}-1\Big)^{r}\leq C\Big(e^{\alpha\beta s^{2}}-1\Big).
\end{equation*}
In particular, if $u\in H^{1}(\mathbb{R}^{2})$ then $\Big(e^{\beta s^{2}}-1\Big)^{r}$ belongs to $L^{1}(\mathbb{R}^{2})$.
\end{lemma}

\begin{lemma}(\emph{see \cite{r9}})
Suppose  $u\in H^{1}(\mathbb{R}^{2})$, $\beta>0$, $q>0$ and $\Vert v\Vert\leq M$ with $\beta M^{2}< 4\pi$, then there exists $C=C(\beta, M, q)>0$ such that
\begin{equation*}
\int_{\mathbb{R}^{2}}\Big(e^{\beta v^{2}}-1\Big)\vert v\vert^{q}dx\leq C\Vert v\Vert^{q}.
\end{equation*}
\end{lemma}

\begin{lemma}
Let $(f_{1})$, $(f_{2})$ hold. Then functional $J$ satisfy the mountain pass geometry:
\begin{description}
  \item[(1)] There exists $\rho>0$ small enough, such that $\underset{\Vert u\Vert=\rho}{\inf}J(u)\geq d>0$;
  \item[(2)] There exists $u_{0}\in  H^{1}_{r}(\mathbb{R}^{2})$ with $\Vert u_{0}\Vert> \rho$, such that $J(u_{0})<0$.
\end{description}
\end{lemma}
\begin{proof}
(1) From $(f_{1})$, for any $\eta>0$, there exists $\delta>0$ such that $\vert u\vert< \delta$
\begin{equation}
 F(x, u)\leq \eta\vert u\vert^{2} \quad \text{for all}\, x\in \mathbb{R}^{2}.
\end{equation}
On the other hand, for $q>2$, by $(f_{2})$, there exists $C=C(q, \delta)$ such that $\vert u\vert\geq \delta$ implies
\begin{equation}
F(x, u) \leq C\vert u\vert^{q}\Big(\exp(4\pi u^{2})-1\Big) \quad \text{for all}\, x\in \mathbb{R}^{2}.
\end{equation}
Combing (2.1) and (2.2) yield
\begin{equation*}
F(x, u)\leq \eta \vert u\vert^{2}+C\vert u\vert^{q}\big(\exp(4\pi u^{2})-1\big) \quad \text{for all}\, (x, u)\in \mathbb{R}^{2}\times \mathbb{R}.
\end{equation*}
So, we have
\begin{align*}
J(u)&\geq \frac{1}{2}\Vert u\Vert^{2}-\int_{\mathbb{R}^{2}}F(x, u)dx\nonumber\\
&\geq  \frac{1}{2}\Vert u\Vert^{2}-\eta\int_{\mathbb{R}^{2}}\vert u\vert^{2}dx-C\int_{\mathbb{R}^{2}}\vert u\vert^{q}\big(\exp(4\pi u^{2})-1\big))dx\\
&\geq  \frac{1}{4}\Vert u\Vert^{2}-C\Big(\int_{\mathbb{R}^{2}}\vert u\vert^{qs}dx\Big)^{\frac{1}{s}}\Big(\int_{\mathbb{R}^{2}}\big(\exp(4\pi u^{2})-1\big))^{r}dx\Big)^{\frac{1}{r}}\\
&\geq  \frac{1}{4}\Vert u\Vert^{2}-C\Big(\int_{\mathbb{R}^{2}}\vert u\vert^{qs}dx\Big)^{\frac{1}{s}}\Big(\int_{\mathbb{R}^{2}}\big(\exp(4\alpha \pi u^{2})-1\big)dx\Big)^{\frac{1}{r}}\\
&\geq  \frac{1}{4}\Vert u\Vert^{2}-C\Vert u\Vert^{q}\geq d>0,
\end{align*}
where $\alpha>1$, $\Vert u\Vert=\rho$ small enough such that $4\pi\alpha\rho<4\pi$, $r>1$ close to 1, $s>1$ and $\frac{1}{r}+\frac{1}{s}=1$, $0<\eta\leq\frac{1}{4}$.\\

(2) Due to  $(f_{4})$, $F(x, u)\geq \frac{C_{p}}{p}\vert u\vert^{p}$. So, for any $u\in  H^{1}_{r}(\mathbb{R}^{2})$, we have
\begin{equation*}
J(u)\leq\frac{1}{2}\int_{\mathbb{R}^{2}}(\vert \nabla u\vert^{2}+u^{2})dx+\frac{\lambda}{2}\int_{\mathbb{R}^{2}}\frac{u^{2}}{\vert x\vert^{2}}\left(\int_{0}^{\vert x\vert}\frac{s}{2}u^{2}(s)\right)^{2}dx-\int_{\mathbb{R}^{2}}\frac{C_{p}}{p}\vert u\vert^{p}dx.
\end{equation*}
Fix $ u\in H^{1}_{r}(\mathbb{R}^{2})\setminus \{0\}$, we have
\begin{align*}
J(tu)\leq\frac{t^{2}}{2}\int_{\mathbb{R}^{2}}\vert \nabla u\vert^{2}+u^{2}dx+\frac{t^{6}}{2}\int_{\mathbb{R}^{2}}\frac{u^{2}}{\vert x\vert^{2}}\left(\int_{0}^{\vert x\vert}\frac{s}{2}u^{2}(s)\right)^{2}dx-\frac{C_{p}t^{p}}{p}\int_{\mathbb{R}^{2}}\vert u\vert^{p}dx.\nonumber
\end{align*}
Since $p>6$, there exists $t_{0}$ sufficiently large such that $\Vert t_{0}u\Vert>\rho$ and $J(t_{0}u)<0$. Set $u_{0}=t_{0}u$, we get the conclusion.
\end{proof}

By the mountain pass theorem (see \cite{r19}), there exists a Palais-Smale
sequence $\{u_{n}\}\subset  H^{1}_{r}(\mathbb{R}^{2})$ satisfying
\begin{align*}
J(u_{n})\rightarrow c\geq b\quad \text{and}\quad J'(u_{n})\rightarrow 0,
\end{align*}
where $c=\underset{\gamma\in\Gamma}{\inf}\underset{t\in [0, 1]}{\max}J(\gamma(t))>0$ and
\begin{align*}
\Gamma=\Big\{\gamma\in C([0, 1],  H^{1}_{r}(\mathbb{R}^{2})): \gamma(0)=0, J(\gamma(1))<0\Big\},
\end{align*}
shortly $\{u_{n}\}$ is a $(PS)_{c}$ sequence. Moreover, by the assumptions of $f$, we may assume that the sequence $\{u_{n}\}$ is nonnegative.

\begin{lemma}
If $(u_{n})\subset H^{1}_{r}(\mathbb{R}^{2})$ is a $(PS)_{c}$ sequence to $J$, then
$(u_{n})$ is bounded in $H^{1}_{r}(\mathbb{R}^{2})$.
\end{lemma}
\begin{proof}
From $(f_{3})$, for $n$ large enough, we have
\begin{align*}
6c+1+\epsilon_{n}\Vert u_{n}\Vert\geq 6J(u_{n})-J'(u_{n})u_{n}&=2\Vert u_{n}\Vert^{2}+\int_{\mathbb{R}^{2}}\Big (u_{n}f(x, u_{n})-6F(x, u_{n})\Big)dx\\
&\geq 2\Vert u_{n}\Vert^{2}-\sigma\int_{\mathbb{R}^{2}}\vert u_{n}\vert^{2}dx\\
&\geq (2-\sigma)\Vert u_{n}\Vert^{2}
\end{align*}
where $\epsilon_{n}\rightarrow 0$, it implies the boundedness of $(u_{n})$.
\end{proof}

\begin{lemma}
Assume that $(f_{1})$ hold, there exists $\lambda_{1}>0$, then for any $0<\lambda<\lambda_{1}$, $d \leq c<\frac{2-\sigma}{6}$.
\end{lemma}

\begin{proof}
According to Lemma 2.5, it is clear that $c\geq b$, so we only need to prove that $c<\frac{2-\sigma}{6}$.\\
Fix a positive function $v_{p}\in H^{1}_{r}(\mathbb{R}^{2})$ such that
\begin{equation*}
S_{p}=\frac{\Bigg(\int_{\mathbb{R}^{2}}\vert \nabla v_{p}\vert^{2}+\vert v_{p}\vert^{2}dx\Bigg)^{\frac{1}{2}}}{\Bigg(\int_{\mathbb{R}^{2}}\vert v_{p}\vert^{p}dx\Bigg)^{\frac{1}{p}}}.
\end{equation*}
Note that
\begin{align*}
\underset{t\geq 0}{\max}J_{0}(tv_{p})&\leq \underset{t\geq 0}{\max}\Bigg\{\frac{t^{2}}{2}\int_{\mathbb{R}^{2}}\vert \nabla v_{p}\vert^{2}+v_{p}^{2}dx-\frac{C_{p}t^{p}}{p}\int_{\mathbb{R}^{2}}\vert v_{p}\vert^{p}dx\Bigg\}\\
&=\frac{(p-2)}{2p}\frac{S_{p}^{\frac{2p}{p-2}}}{C_{p}^{\frac{2}{p-2}}}.\nonumber
\end{align*}
So, there exists $\lambda_{1}>0$ such that for any $0<\lambda<\lambda_{1}$, we have
\begin{equation*}
\underset{t\geq 0}{\max}J(tv_{p})\leq \frac{(p-2)}{p}\frac{S_{p}^{\frac{2p}{p-2}}}{C_{p}^{\frac{2}{p-2}}}.
\end{equation*}
Moreover, from $(f_{4})$
\begin{equation*}
\frac{(p-2)}{p}\frac{S_{p}^{\frac{2p}{p-2}}}{C_{p}^{\frac{2}{p-2}}}< \frac{2-\sigma}{6}.
\end{equation*}
So $c<\frac{2-\sigma}{6}$.
\end{proof}

\begin{proof}[\textbf{Proof of Theorem 1.1}]
From Lemma 2.5, we obtain a  nonnegative $(PS)_{c}$  sequence $\{u_{n}\}$ and from Lemma 2.6, this sequence is bounded, thus
for a subsequence still denoted by $\{u_{n}\}$ there is a nonnegative function $u_{1}\in H^{1}_{r}(\mathbb{R}^{2})$ such that $u_{n}\rightharpoonup u_{1}$ in
 $H^{1}_{r}(\mathbb{R}^{2})$ and $u_{n}\rightarrow u_{1}$ in $L_{loc}^{s}(\mathbb{R}^{2})$ for all $s\geq 1$ and $u_{n}\rightarrow u_{0}$ almost everywhere
 in $\mathbb{R}^{2}$. Now, from $(f_{3})$,
\begin{align*}
c&=\underset{n\rightarrow\infty}{\lim}J(u_{n})\\
&=\underset{n\rightarrow\infty}{\lim}\left[J(u_{n})-\frac{1}{6}J'(u_{n})u_{n}\right]\nonumber\\
&\geq \underset{n\rightarrow\infty}{\lim}\left(\frac{1}{3}\Vert u_{n}\Vert^{2}+\int_{\mathbb{R}^{2}}\left(\frac{1}{6}f(x, u_{n})u_{n}-F(x, u_{n})\right)dx\right)\\
&\geq \frac{2-\sigma}{6}\underset{n\rightarrow\infty}{\lim\sup}\Vert u_{n}\Vert^{2}\\
\end{align*}
which implies
\begin{equation*}
\underset{n\rightarrow\infty}{\lim\sup}\Vert u_{n}\Vert^{2}=m\leq \frac{6 c}{2-\sigma}< 1.
\end{equation*}
According to Trudinger-Moser inequality (1.5) and $(f_{1})$, for any bounded domain $\Omega$, $f(x, u_{1})$ and $f(x, u_{n})$ belong to $L^{1}(\Omega)$.
In virtue of $J'( u_{n}) u_{n}\rightarrow 0$ and $J(u_{n})\rightarrow c$ as $n\rightarrow\infty$, there exists a constant $C_{2}>0$ such that
\begin{equation*}
\int_{\mathbb{R}^{2}}F(x, u_{n})dx\leq C_{2} \quad \text{and}\quad \int_{\mathbb{R}^{2}}f(x, u_{n})u_{n}dx\leq C_{2}.
\end{equation*}
By $(f_{1})$, we have
\begin{align*}
\int_{\Omega} f(x, u_{n})u_{n} dx\leq C_{2}.
\end{align*}
Therefore, according to Lemma 2.2, one has
\begin{equation}
\underset{n\rightarrow \infty}{\lim}\int_{\mathbb{R}^{2}}f(x, u_{n})vdx=\int_{\mathbb{R}^{2}}f(x, u_{1})vdx\quad\quad \forall\, \varphi\in C^{\infty}_{0, r}(\mathbb{R}^{2}).
\end{equation}
Combing (2.3) and Lemma 2.1, we have
\begin{equation*}
\underset{n\rightarrow \infty}{\lim}J'(u_{n})v=0=J'(u_{1})v \quad\quad \forall\, v\in C^{\infty}_{0, r}(\mathbb{R}^{2}),
\end{equation*}
so $u_{0}$ is a solution of problem (1.6). \\

At last, we show that the sequence $(u_{n})$ has a convergent subsequence.\\

Set $u_{n}=u_{1}+\omega_{n}$, then $\omega_{n}\rightharpoonup 0$ in $H^{1}_{r}(\mathbb{R}^{2})$ and $\omega_{n}\rightharpoonup 0$ in $L^{r}(\mathbb{R}^{2})$ for all
$2<q<\infty$. By the Br\'{e}zis-Lieb Lemma (see \cite{r19}), we get
\begin{equation}
\Vert u_{n}\Vert^{2}=\Vert u_{1}\Vert^{2}+\Vert \omega_{n}\Vert^{2}+o_{n}(1).
\end{equation}
We firstly show that
\begin{equation}
\underset{n\rightarrow\infty}{\lim}\int_{\mathbb{R}^{2}}f(x, u_{n})u_{1}dx=\int_{\mathbb{R}^{2}}f(x, u_{1})u_{1}dx.
\end{equation}
In fact, since $C_{0, r}^{\infty}(\mathbb{R}^{2})$ is dense in $H^{1}_{r}(\mathbb{R}^{2})$, for any $\eta>0$ there exists $\varphi\in C_{0, r}^{\infty}(\mathbb{R}^{2})$
such that $\Vert u_{1}-\varphi\Vert<\eta$. Observe that
\begin{align*}
\Big\vert\int_{\mathbb{R}^{2}}f(x, u_{n})u_{1}dx-\int_{\mathbb{R}^{2}}f(x, u_{1})u_{1}dx\Big\vert &\leq\Big\vert\int_{\mathbb{R}^{2}}f(x, u_{n})(u_{1}-\varphi)dx\Big\vert+\Big\vert\int_{\mathbb{R}^{2}}f(x, u_{1})(u_{1}-\varphi) dx\Big\vert\\
&+\Vert \varphi\Vert_{\infty}\int_{\text{supp} \varphi}\Big\vert f(x, u_{n})-f(x, u_{1})\Big\vert dx.
\end{align*}
For the first integral, using that $\vert J'(u_{n})(u_{1}-\varphi)\vert\leq\eta_{n}\Vert u_{1}-\varphi\Vert$ with $\eta_{n}\rightarrow 0$
as $n\rightarrow \infty$ and Lemma 2.1, we get
\begin{align*}
\Big\vert\int_{\mathbb{R}^{2}}f(x, u_{n})(u_{1}-\varphi)dx\Big\vert &\leq\eta_{n}\Vert u_{1}-\varphi\Vert+\Big\vert\int_{\mathbb{R}^{2}}\nabla u_{n} \nabla(u_{1}-\varphi)+u_{n}(u_{1}-\varphi)dx\Big\vert+\vert c'(u_{n})( u_{1}-\varphi)\Big\vert\\
&\leq\eta_{n}\Vert u_{1}-\varphi\Vert+\Vert u_{n}\Vert \Vert  u_{1}-\varphi\Vert +\Big\vert \Big(c'(u_{n})-c'(u_{1})\Big)( u_{1}-\varphi)\Big\vert+\Big\vert c'(u_{1})( u_{1}-\varphi)\Big\vert\\
&\leq C_{3}\Vert u_{1}-\varphi\Vert\leq C_{3}\eta
\end{align*}
for $n$ large. Similarly, using that  $ J'(u_{1})(u_{1}-\varphi)=0$, we can estimate the second integral and obtain
\begin{align*}
\Big\vert\int_{\mathbb{R}^{2}}f(x, u_{1})(u_{1}-\varphi) dx\Big\vert\leq C_{3}\eta.
\end{align*}
Combing (2.3) and the previous inequality, we have
\begin{equation*}
\underset{n\rightarrow\infty}{\lim}\Big\vert \int_{\mathbb{R}^{2}}f(x, u_{n})u_{1}dx-\int_{\mathbb{R}^{2}}f(x, u_{1})u_{1}dx\Big\vert\leq 2C_{3}\eta,
\end{equation*}
this implies (2.5) because $\eta$ is arbitrary.\\
From (2.4) and Lemma 2.1, we can write
\begin{align*}
J'(u_{n})u_{n}&=\Vert u_{n}\Vert^{2}+c'(u_{n})u_{n}-\int_{\mathbb{R}^{2}}f(x, u_{n})u_{n}dx\\
&=\Vert u_{1}\Vert^{2}+\Vert \omega_{n}\Vert^{2}+c'(u_{1})u_{1}-\int_{\mathbb{R}^{2}}f(x, u_{n})u_{1}dx-\int_{\mathbb{R}^{2}}f(x, u_{n})\omega_{n}dx+o_{n}(1)\\
&=J'(u_{1})u_{1}+\Vert \omega_{n}\Vert^{2}-\int_{\mathbb{R}^{2}}f(x, u_{n})\omega_{n}dx+o_{n}(1),
\end{align*}
that is
\begin{align*}
\Vert \omega_{n}\Vert^{2}=\int_{\mathbb{R}^{2}}f(x, u_{n})\omega_{n}dx+o_{n}(1).
\end{align*}
According to Trudinger-Moser inequality (1.5), for $\tau>1$, $q>1$ close to 1 satisfying $\tau q\frac{(\mu-2)c}{2\mu}<1$, there exists $C_{4}>0$
such that the sequence $h_{n}(x)=e^{4\pi\tau u_{n}^{2}(x)}-1$ satisfies
\begin{align*}
\int_{\mathbb{R}^{2}}f(x, u_{n})\omega_{n}dx&\leq \eta\int_{\mathbb{R}^{2}}\vert u_{n}\omega_{n}\vert dx+C_{\eta}\int_{\mathbb{R}^{2}}\Big(e^{4\pi\tau u_{n}^{2}(x)}-1\Big)\vert \omega_{n}\vert dx\\
&\leq \eta +C_{4}\Bigg(\int_{\mathbb{R}^{2}}\vert \omega_{n}\vert^{q'}dx\Bigg)^{\frac{1}{q'}}\Bigg(\int_{\mathbb{R}^{2}}(e^{4\pi(\sqrt{\tau q} u_{n}(x))^{2}}-1)dx\Bigg)^{\frac{1}{q}}\\
&\leq \eta+C_{4}\Bigg(\int_{\mathbb{R}^{2}}\vert \omega_{n}\vert^{q'}dx\Bigg)^{\frac{1}{q'}}\nonumber
\end{align*}
where $\frac{1}{q}+\frac{1}{q'}=1$. Since $q>1$ close to 1, $q'>2$, by the compact embedding $H^{1}_{r}(\mathbb{R}^{2})\hookrightarrow L^{r}(\mathbb{R}^{2})$ for all $r>2$, we get
\begin{align*}
\int_{\mathbb{R}^{2}}f(x, u_{n})\omega_{n}dx\rightarrow 0, \quad \text{as}\, n\rightarrow\infty.
\end{align*}
So, $\underset{n\rightarrow\infty}{\lim}\Vert \omega_{n}\Vert^{2}=0$. Moreover, by the argument in \cite{r4},  $u_{1}\in C^{2}(\mathbb{R}^{2})$. Since
$u_{1}$ is nonnegative, we have $u_{1}>0$ by the strong maximum principle and the proof is completed.\\
\end{proof}

\section{Proof of Theorem 1.2}
In this section, we deal with the problem (0.1) and show that there exist at least two nonnegative solutions, one is mountain-pass type solution,
another is a local minimal solution with negative energy. \\

The functional corresponding to problem (1.1) is
\begin{align*}
J_{\epsilon}(u)&=\frac{1}{2}\int_{\mathbb{R}^{2}}(\vert \nabla u\vert^{2}+u^{2})dx+\frac{\lambda}{2}\int_{\mathbb{R}^{2}}\frac{u^{2}}{\vert x\vert^{2}}\left(\int_{0}^{\vert x\vert}\frac{s}{2}u^{2}(s)\right)^{2}dx-\int_{\mathbb{R}^{2}}F(x, u)dx-\epsilon\int_{\mathbb{R}^{2}}kudx\\
&=\frac{1}{2}\int_{\mathbb{R}^{2}}(\vert \nabla u\vert^{2}+u^{2})dx+\frac{\lambda}{2}\int_{\mathbb{R}^{2}}\frac{u^{2}}{4\vert x\vert^{2}}\left(\frac{1}{2\pi}\int_{B_{\vert x\vert}}u^{2}\right)^{2}dx-\int_{\mathbb{R}^{2}}F(x, u)dx-\epsilon\int_{\mathbb{R}^{2}}kudx\\
\end{align*}
for $u\in H^{1}_{r}(\mathbb{R}^{2})$. It is easy to show that $J_{\epsilon}\in C^{1}(H^{1}_{r}(\mathbb{R}^{2}), \mathbb{R})$ with
\begin{equation}
J_{\epsilon}'(u)\phi=\int_{\mathbb{R}^{2}} \nabla u\nabla \varphi+u\varphi dx+\lambda\int_{\mathbb{R}^{2}}\left(\int_{\vert x\vert}^{\infty}\frac{h(s)}{s}u^{2}(s)ds\right)u\varphi+\frac{h^{2}(\vert x\vert)}{\vert x\vert^{2}}u\varphi dx-\int_{\mathbb{R}^{2}}f(x, u)\varphi dx -\epsilon\int_{\mathbb{R}^{2}}k\varphi dx
\end{equation}
for any $\varphi\in H^{1}_{r}(\mathbb{R}^{2})$. So, for searching the solutions of problem (0.1), we may seek the critical points of the functional $J_{\epsilon}$.\\

In the next two lemmas we check that the functional $J_{\epsilon}$ satisfies the geometric
conditions of the mountain-pass theorem.
\begin{lemma}
Let $(f_{1})$, $(f_{2})$ hold. Then there exists $\epsilon_{1}>0$ such that for $0<\epsilon<\epsilon_{1}$, there exists $\rho_{\epsilon}>0$
such that $J_{\epsilon}(u)>0$ if  $\Vert u\Vert=\rho_{\epsilon}$. Furthermore, $\rho_{\epsilon}$ can be chosen such that $\rho_{\epsilon}\rightarrow 0$
as $\epsilon\rightarrow 0$.
\end{lemma}

\begin{proof}
As the same proof of Lemma 2.5, from $(f_{1})$ and $(f_{2})$, for $\forall 0<\eta<\frac{1}{4}$, there exist $C>0$ such that for $q>2$
\begin{equation*}
F(x, u) \leq \eta\vert u\vert^{2}+C\vert u\vert^{q}\Big(\exp(4\pi u^{2})-1\Big) \quad \text{for all}\, x\in \mathbb{R}^{2}\times \mathbb{R}.
\end{equation*}
So, by Lemma 2.4, we have
\begin{align*}
J_{\epsilon}(u)&\geq \frac{1}{2}\Vert u\Vert^{2}-\int_{R^{2}}F(x, u)dx-\epsilon\int_{R^{2}}k(x)udx\nonumber\\
&\geq  \frac{1}{2}\Vert u\Vert^{2}-\eta\int_{R^{2}}\vert u\vert^{2}dx-C\int_{R^{2}}\vert u\vert^{q}\left(\exp(4\pi u^{2})-1\right)dx-\epsilon\Vert k\Vert_{*}\Vert u\Vert\\
&\geq  \frac{1}{4}\Vert u\Vert^{2}-C\Big(\int_{R^{2}}\vert u\vert^{qs}dx\Big)^{\frac{1}{s}}\left(\int_{R^{2}}\left(\exp(4\pi u^{2})-1\right)^{r}dx\right)^{\frac{1}{r}}-\epsilon\Vert k\Vert_{*}\Vert u\Vert\\
&\geq  \frac{1}{4}\Vert u\Vert^{2}-C\Vert u\Vert^{q}-\epsilon\Vert k\Vert_{*}\Vert u\Vert, \quad \text{for}\,\, \Vert u\Vert=\rho\,\, \text{small enough}
\end{align*}
where $r>1$ close to 1, $s>1$ and $\frac{1}{r}+\frac{1}{s}=1$. Thus
\begin{align}
J_{\epsilon}(u)\geq\Vert u\Vert\left(\frac{1}{4}\Vert u\Vert-C\Vert u\Vert^{q-1}-\epsilon\Vert k\Vert_{*}\right).
\end{align}
 Since $q>2$, we may choose $\rho>0$ small enough such that $\frac{1}{4}\rho-C\rho^{q-1}>0$.
Thus, if $\epsilon>0$ is sufficiently small then we can find some $\rho_{\epsilon}>0$ such that $J_{\epsilon}(u)>0$ if  $\Vert u\Vert=\rho_{\epsilon}$ and $\rho_{\epsilon}\rightarrow 0$ as $\epsilon\rightarrow 0$.
\end{proof}

\begin{lemma}
There exists $u_{0}\in  H^{1}_{r}(\mathbb{R}^{2})$ with $\Vert u_{0}\Vert> \rho_{\epsilon}$, such that $J_{\epsilon}(u_{0})<\underset{\Vert u\Vert=\rho_{\epsilon}}{\inf}J_{\epsilon}(u)$.
\end{lemma}
\begin{proof}
Due to  $(f_{4})$, as the same proof of Lemma 2.1, we may show $J_{\epsilon}(tu)\rightarrow -\infty$ as $t\rightarrow \infty$. Setting
$u_{0}=tu$ with $t$ sufficiently large, we get the result.
\end{proof}

From Lemma 3.1 and Lemma 3.2, we can get a $(PS)_{c_{\epsilon}}$ sequence $\{u_{n}\}\subset  H^{1}_{r}(\mathbb{R}^{2})$ satisfying
\begin{align*}
J_{\epsilon}(u_{n})\rightarrow c_{\epsilon}>0\quad \text{and}\quad J_{\epsilon}'(u_{n})\rightarrow 0,
\end{align*}
where $c_{\epsilon}=\underset{\gamma\in\Gamma}{\inf}\underset{t\in [0, 1]}{\max}J_{\epsilon}(\gamma(t))>0$ and
\begin{align*}
\Gamma=\Big\{\gamma\in C([0, 1],  H^{1}_{r}(\mathbb{R}^{2})): \gamma(0)=0, J_{\epsilon}(\gamma(1))<0\Big\}.
\end{align*}

\begin{lemma}
 Assume that $(f_{1})$ hold, and let $\lambda_{1}>0$ be as in Lemma 2.7,  then there exists $0<\epsilon_{2}<\epsilon_{1}$ such that for any $0<\lambda<\lambda_{1}$ and $0<\epsilon<\epsilon_{1}$, $c_{\epsilon}<\frac{2-\sigma}{6}$.
\end{lemma}

\begin{lemma}
Suppose $(f_{1})-(f_{4})$ hold, and let $\lambda_{1}>0$ and $\epsilon_{2}>0$ be as in Lemma 3.3. Then for any $0<\lambda<\lambda_{1}$ and $0<\epsilon<\epsilon_{2}$, problem $(0.1)$
has a mountain pass type solution $u_{2}$.
\end{lemma}

The proof of Lemma 3.3 is similar to one of Lemma 2.7, the proof of Lemma 3.4 is similar to one of Theorem 1.1, so we omit them.\\

Now we prove the existence of a local minimal solution with the negative energy.
\begin{lemma}
There exists $\eta>0$ and $v\in H^{1}_{r}(\mathbb{R}^{2})$ with $\Vert v\Vert=1$ such that $J_{\epsilon}(tv)<0$ for all $0<t<\theta$.
In particular, $\underset{\Vert u\Vert=\theta}{\inf}J_{\epsilon}(u)<0$.
\end{lemma}
\begin{proof}
For each $k\in H^{-1}$, borrowing the Riesz representation theorem in the Hilbert space $H^{1}_{r}(\mathbb{R}^{2})$, the problem
\begin{align*}
-\Delta v+v=k, \quad\quad x\in \mathbb{R}^{2},
\end{align*}
has a unique weak solution $v$ in $H^{1}_{r}(\mathbb{R}^{2})$. Thus,
\begin{align*}
\int_{\mathbb{R}^{2}}kvdx=\Vert v\Vert^{2}>0 \quad\quad \text{for each}\,\, k\neq 0.
\end{align*}
Since $f(x, 0)=0$, by continuity, there exists $\theta>0$ such that
\begin{align*}
\frac{d}{dt}J_{\epsilon}(tv)=t\Vert v\Vert^{2}+3t^{5}c(v)-\int_{\mathbb{R}^{2}}f(x, tv)vdx-\epsilon\int_{\mathbb{R}^{2}}kvdx<0
\end{align*}
for all $0<t<\theta$. Since $J_{\epsilon}(0)=0$, it is clear that $J_{\epsilon}(tv)<0$ for all $0<t<\theta$.
\end{proof}

By inequality (3.2) and Lemma 3.5, one has
\begin{align}
-\infty<c_{1}=\underset{\Vert u\Vert\leq \rho_{\epsilon}}{\inf}J_{\epsilon}(u)<0.
\end{align}

\begin{lemma}
Le $\epsilon_{2}>0$ be as in Lemma 3.3 and  each $\epsilon$ with $0<\epsilon<\epsilon_{2}$, problem $(0.1)$ has a minimal type solution $u_{3}$
with  $J_{\epsilon}(u_{3})=c_{0}<0$
\end{lemma}
\begin{proof}
Let $\rho_{\epsilon}$ small be as in Lemma 3.1. We can choose that  $\rho_{\epsilon}<1$.
Since $\bar{B}_{\rho_{\epsilon}}$ is a complete metric space with the metric given by the norm of $H^{1}_{r}(\mathbb{R}^{2})$, convex
and the functional $J_{\epsilon}$ is a class of $C^{1}$ and bounded below on $\bar{B}_{\eta_{\epsilon}}$, by Ekeland's variational principle \cite{r10}
there exists a sequence $\{u_{n}\}$ in $\bar{B}_{\rho_{\epsilon}}$ such that
\begin{align*}
J_{\epsilon}(u_{n})\rightarrow c_{1}<0  \quad\text{and}\quad J'_{\epsilon}(u_{n})\rightarrow 0.
\end{align*}
By the proof of Theorem 1.1 it follows that there exists a subsequence of $\{u_{n}\}$ which converges to
a function $u_{3}$ . Therefore, $J_{\epsilon}(u_{3})=c_{1}<0$.

\end{proof}
\begin{proof}[\textbf{Proof of Theorem 1.2}]
From Lemma 3.4 and Lemma 3.6, there exist $\lambda_{1}>0$ and $\epsilon_{2}>0$ such that for any $0<\lambda<\lambda_{1}$ and $0<\epsilon<\epsilon_{2}$, there exist at least two solutions of problem (0.1), one is a mountain pass type solution, another is a local minimum solution with negative energy. Since $k(x)\geq 0$ almost everywhere in $\mathbb{R}^{2}$. Let $u\in H^{1}_{r}(\mathbb{R}^{2})$ be a weak solution of (0.1). Setting $u^{+}=\max\{u, 0\}$, $u^{-}=\max\{-u, 0\}$ and taking $v=u^{-}$ in (3.1), we obtain
\begin{align*}
\Vert u^{-}\Vert^{2}+3\lambda c(u^{-})=\epsilon \int_{\mathbb{R}^{2}}kvdx\leq 0,
\end{align*}
because $f(x, u(x))u^{-}=0$ in $\mathbb{R}^{2}$. So, $u=u^{+}\geq 0$.
We get two solutions are nonnegative, by the argument in \cite{r4},  these two solutions belong to $\in C^{2}(\mathbb{R}^{2})$. Moreover,
by the strong maximum principle, they are positive. We complete the proof.
\end{proof}

\section*{Acknowledgements}
The first author is supported by NSFC (No. 11301181) and China Postdoctoral Science Foundation funded project. And  the second author is supported by by Young Teachers Foundation  of BTBU (No. QNJJ2016-15).


\section*{References}

\begin{thebibliography}{10}

\addtolength{\itemsep}{-0.5 em} 
\setlength{\itemsep}{-1pt}


\bibitem{r1}H. Berestycki, T. Gallou\"{e}t, O. Kavian, \'{E}quations de Champs scalaires euclidiens non lin\'{e}aires dans le plan, C R. Acad. Sci. Paris S\'{e}r. I Math. 297(1983), 307-310  and Publications du Laboratoire d¡¯Analyse Num\'{e}rique, Universit\'{e} de Paris VI (1984).


\bibitem{r2}H. Berestycki, P.L. Lions, Nonlinear scalar field equations. I. Existence of a ground state, Arch. Rational Mech. Anal. 82(1983), 313-345.

\bibitem{r3}H. Berestycki, P.L. Lions,  Nonlinear scalar field equations. II. Existence of infinitely many solutions, Arch. Rational Mech. Anal. 82(1983), 347-375.

\bibitem{r4}J. Byeon, H. Huh, J. Seok,  Standing waves of nonlinear Schr\"{o}dinger equations with the gauge field, J. Funct. Anal. 263(2012), 1575-1608.

\bibitem{r5}D.M. Cao, Nontrivial solution of semilinear elliptic equation with critical exponent in $R^{2}$, Comm. Partial Differential Equations, 17(1992), 407-435.

\bibitem{r6}P, Cunha, P. d'Avenia, A. Pomponia, G.Siciliano, A multiplicity result for Chern-Simons-Schr\"{o}dinger equation with a general nonlinearity,
Nonlinear Differ. Equ. Appl. 22(2015), 1831-1850.


\bibitem{r7}D.G. de Figueiredo, O.H. Miyagaki, B.Ruf, Elliptic equations in $R^{2}$ with nonlinearities in the critical growth range,  Calc. Var. Partial Differential Equations, 3(1995), 139-153.

\bibitem{r8}J.M. Bezerra do \'{O}, N-Laplacian equations in $R^{N}$ with critical growth, Abstr. Appl. Anal. 2(3-4)(1997), 301-315.

\bibitem{r9}J.M. Bezerra do \'{O}, E. Medeiros, U. Severo, A nonhomogeneous elliptic problem involving critical growth in dimension two, J. Math. Anal. Appl. 345 (2008), 286-304.

\bibitem{r10}I. Ekeland, On the variational principle, J. Math. Anal. Appl. 17(1974), 324-35.


\bibitem{r11}R. Jackiw, S.-Y. Pi, Classical and quantal nonrelativistic Chern-Simons theory, Phys.Rev.D 42(1990), 3500-3513.


\bibitem{r12}R. Jackiw, S.Y. Pi, Self-dual Chern-Simons solitons, Progr. Theoret. Phys. Suppl. 107(1992), 1-40.

\bibitem{r13}Y.S. Jiang, A. Pomponio, D. Ruiz, Standing waves for a gauged nonlinear Schr\"{o}dinger equation with a vortex point, Commun. Contemp. Math. 18(2016), doi:10.1142/S0219199715500741.

\bibitem{r14}J. Moser, A sharp form of an inequality by N. Trudinger, Indiana Univ. Math. J. 20(1971), 1077-1092.

\bibitem{r15}A. Pomponio, D. Ruiz, A variational analysis of a gauged nonlinear Schr\"{o}dinger equation, J. Eur. Math. Soc. (JEMS), 17(6)(2015), 1463-1486.

\bibitem{r16}J. Seok, Infinitely Many Standing Waves for the Nonlinear Chern-Simons-Schr\"{o}dinger Equation, Adv. Math.Phys. 2015(2015), Articale ID 519374,7 pages.

\bibitem{r17}N.S. Trudinger, On the imbedding into Orlicz spaces and some applications, J. Math. Mech. 17(1967), 473-484.

\bibitem{r18}Y.Y. Wan, J.G. Tan, Standing waves for the Chern-Simons-Schr\"{o}dinger systems without (AR) condition, J. Math. Anal. Appl. 415(2014), 422-434.

\bibitem{r19}M. Willem, Minimax Theorems, Birkh\"{a}user, Boston, 1996.

\bibitem{r20}J.J. Yuan,  Mutiple normalized solutions of Chern-Simons-Schr\"{o}dinger system, Nonlinear Differ. Equ. Appl. 22(2015), 1801-1816.

\end{thebibliography}
\end{document}